\documentclass[12pt, reqno]{amsart}
\usepackage[dvipsnames]{xcolor}
\usepackage{mathtools}
\mathtoolsset{showonlyrefs}
\usepackage{etoolbox}
\usepackage{amscd}
\usepackage{verbatim}
\usepackage{stmaryrd}
\usepackage{bbm}
\usepackage{enumitem} 
\usepackage{blindtext}
\usepackage{latexsym}
\DeclareMathOperator\erfc{erfc}
\usepackage{exscale}
\usepackage{amsmath}
\usepackage{mathabx}
\usepackage{amssymb}
\usepackage{amsfonts}
\usepackage{mathrsfs}
\usepackage{amsbsy}
\usepackage{relsize}
\usepackage{amsxtra}
\usepackage{microtype}  
\usepackage{parskip} 

\usepackage{capt-of}
\usepackage{lipsum}
\usepackage{parskip} 
\usepackage{graphicx}
\usepackage{epstopdf} 
\usepackage{float}
\usepackage{xspace}
\usepackage{lmodern}

\usepackage{comment}
\PassOptionsToPackage{reqno}{amsmath}
\usepackage[margin=1in]{geometry}
\definecolor{mypink1}{rgb}{0.858, 0.188, 0.478}
\definecolor{mypink2}{RGB}{219, 48, 122}
\definecolor{mypink3}{cmyk}{0, 0.7808, 0.4429, 0.1412}
\definecolor{mygray}{gray}{0.6}

\definecolor{gris75}{gray}{0.25}
\definecolor{violet}{rgb}{0.5,0,0.5}

\definecolor{BrickRed}{rgb}{0.58, 0.0, 0.83}

\definecolor{armygreen}{rgb}{0.29, 0.33, 0.13}

\definecolor{brass}{rgb}{0.71, 0.65, 0.26}
\definecolor{antiquefuchsia}{rgb}{0.57, 0.36, 0.51}
\definecolor{amethyst}{rgb}{0.6, 0.4, 0.8}

\definecolor{mauvetaupe}{rgb}{0.57, 0.37, 0.43}

\usepackage[colorlinks, linkcolor=purple, citecolor=blue, urlcolor=blue, hypertexnames=false]{hyperref}

\newcommand{\R}{\ensuremath{\mathbb{R}}}

\newcommand{\N}{\ensuremath{\mathbb{N}}}
\newcommand{\C}{\ensuremath{\mathbb{C}}}
\newcommand{\E}{\ensuremath{\mathbb{E}}}
\newcommand{\Var}{\ensuremath{\operatorname{Var}}}
\newcommand{\bz}{\ensuremath{\mathbf{Z}}}
\newcommand{\bb}{\ensuremath{\mathbf{B}}}
\newcommand{\bj}{\ensuremath{\mathbf{J}}}
\newcommand{\bw}{\ensuremath{\mathbf{W}}}

\newtheorem{definition}{Definition}[section]
\newtheorem{proposition}{Proposition}[section]
\newtheorem{theorem}{Theorem}[section]
\newtheorem{remark}{Remark}[section]
\newtheorem{lemma}{Lemma}[section]

\title[Stochastic PDEs]{Mild Solutions for Time--Fractional Stochastic Nonlocal Diffusion Equations}
\author[M. Alwohaibi, D. Alsaleh, M. El-Beltagy, M. Majdoub, E. Mliki]{\tiny Maram Alwohaibi, Dana Alsaleh, Mohamed El-Beltagy, Mohamed Majdoub, Ezzedine Mliki}

\address[M. Alwohaibi, D. Alsaleh, M. Majdoub, E. Mliki]{\small Department of Mathematics, College of Science, Imam Abdulrahman Bin Faisal University\\ P. O. Box 1982, Dammam, Saudi Arabia}
\address[M. Alwohaibi, D. Alsaleh, M. Majdoub, E. Mliki]{\small Basic and Applied Scientific Research Center, Imam Abdulrahman Bin Faisal University\\ P.O. Box 1982, 31441, Dammam, Saudi Arabia}
\address[Mohamed El-Beltagy]{\small Engineering Mathematics and Physics Dept., Faculty of Engineering, Cairo University, Egypt}
\email{\sl \color{blue}{malwohaibi@iau.edu.sa}}
\email{\sl \color{blue}{dalsaleh@iau.edu.sa}}
\email{\sl \color{blue}{zbeltagy@eng.cu.edu.eg}}
\email{\sl {\color{blue}{mmajdoub@iau.edu.sa}}}
\email{\sl \color{blue}{ermliki@iau.edu.sa}}

\begin{document}
 
\begin{abstract}
We study a time--space nonlocal diffusion equation driven by additive time--space white noise, where the time derivative is the Caputo derivative of order $\alpha\in(0,2)$. The model couples local diffusion with a nonlocal convolution operator generated by a radial probability density, thus incorporating memory effects and long-range spatial interactions.

For Dirac initial data, we derive an explicit solution formula in the space of tempered distributions, decomposing the solution into a deterministic part and a stochastic convolution kernel expressed through Mittag--Leffler functions. Our main contribution is a sharp characterization of the existence of mild solutions in terms of $\alpha$, the spatial dimension $N$, and the coefficients of the local and nonlocal diffusion terms. In particular, when the Laplacian term is absent, no mild solution exists, whereas for $\lambda>0$ the admissible regimes depend critically on $(\alpha,N)$, extending and sharpening the known results for purely local fractional stochastic heat equations. Numerical simulations illustrate the evolution of the mean and variance and emphasize subdiffusive spreading and memory effects.
\end{abstract}




\date{\today}

\maketitle


\section{Introduction}
\label{S1}

Fractional-in-time evolution equations have become a standard tool for modeling \emph{anomalous diffusion} and
transport phenomena with memory and hereditary effects.  Historically, fractional operators already appear
in Abel's work on integral equations, and the Mittag--Leffler function---introduced by Mittag--Leffler in 1903---is
now recognized as the natural analogue of the exponential function for fractional dynamics.  In particular,
for $0<\alpha<1$ the Caputo derivative $\partial_t^\alpha$ typically produces \emph{subdiffusive} behavior, while
$\alpha>1$ leads to regimes that are often associated with faster-than-classical propagation and wave-like memory.

When random perturbations are present, fractional stochastic partial differential equations provide a flexible and
natural framework for modeling systems subject to uncertainty and long-range temporal dependence; see, for instance,
the general background in \cite{GM1998,Gorenflo2020,Podlubny1999,KST}.
In the presence of time--space white noise, a convenient and rigorous approach is to interpret the noise in the
distributional sense within the Hida space $(\mathscr{S})^\ast$, which allows one to work with solutions taking values
in the space of tempered distributions.
Within this framework, stochastic integrals with deterministic integrands can be defined via the
It\^o--Skorohod integral and Wick multiplication, yielding a powerful calculus for stochastic evolution equations;
see \cite{Benth,HidaBook,LindstromUboe}.
This approach has proved particularly effective in the analysis of fractional stochastic heat-type equations driven by
time--space white noise, as developed recently in \cite{Hach}.

In this paper we study a \emph{time--space nonlocal} fractional diffusion model driven by additive time--space white noise:
\begin{equation}
    \label{Main-eq}
    \partial^\alpha_t \bz(t,x)= \lambda \Delta \bz(t,x) + \mu\Big(\mathbf{J}\ast\bz(t,x)-\bz(t,x)\Big)+\sigma \bw(t,x),
\end{equation}
where $\alpha\in(0,2)$, $\lambda,\mu\ge0$, $\sigma\in\mathbb{R}$, and $\bj$ is a radial probability density on
$\mathbb{R}^N$.  The Laplacian models local diffusion, whereas the nonlocal term $\mu\Big(\mathbf{J}\ast\bz-\bz\Big)$ accounts for
long-range spatial interactions (jump-like mixing), which naturally arises in models of dispersal and nonlocal
transport.

Fractional stochastic heat equations without spatial nonlocality ($\mu=0$) have been investigated in depth; in
particular, Moulay Hachemi and \O ksendal \cite{Hach} obtained existence and partial mildness criteria for the purely local
case, highlighting how the admissible spatial dimensions depend sharply on $\alpha$.
However, the interplay between fractional time dynamics, \emph{nonlocal} spatial diffusion, and stochastic forcing
is less understood, and the influence of the nonlocal term on the key space--time integrability properties of the
stochastic convolution has not been fully clarified.

The first goal of the present work is to derive an explicit representation formula for the solution of
\eqref{Main-eq} under singular initial data,
\begin{equation}\label{eq:IC-Dirac}
Z(0,x)=\delta_0(x),
\end{equation}
together with the natural decay condition at infinity and, when $\alpha\in(1,2)$, the standard additional Caputo
compatibility condition $\partial_t \bz(0,x)=0$.
By combining Laplace and Fourier transform techniques with sharp properties of Mittag--Leffler functions, we obtain
a decomposition $\bz=\bz_1+\bz_2$ into a deterministic component and a stochastic convolution term, yielding a natural notion
of mild solution in the tempered-distribution setting.

Our second and main objective is to provide \emph{sharp} existence criteria for mild solutions in terms of the
parameters $(\alpha,N,\lambda,\mu)$.  The key difficulty lies in the stochastic part: by It\^o isometry and
Plancherel's theorem, mildness reduces to delicate space--time integrability problems involving squares of
Mittag--Leffler kernels.  We establish optimal conditions based on the precise asymptotic behavior of these special
functions, leading to a complete classification of admissible regimes.

A striking consequence is that if the local diffusion is absent ($\lambda=0$ and $\mu>0$), then \emph{no mild solution}
exists in any spatial dimension and for any $\alpha\in(0,2)$.  In contrast, when $\lambda>0$ the Laplacian dominates
the nonlocal operator at high frequencies, and mild solutions may exist; the admissible dimensions depend strongly on
$\alpha$ (see Theorem~\ref{Thm2}).
Finally, we complement the analysis with numerical simulations illustrating the behavior of the mean and the variance
in representative regimes, with emphasis on subdiffusive spreading and memory effects induced by the fractional time
derivative.
These results extend and sharpen earlier ones for the purely local case and highlight the subtle role played by fractional time derivatives.

Our first main result can be stated as follows
\begin{theorem}
    \label{Thm1}
The unique solution $\mathbf{Z}(t,x)\in \mathcal{S}'$ of \eqref{Main-eq}-\eqref{eq:IC-Dirac} is given by
\begin{equation}\label{Sol}
\begin{aligned}
\bz(t,x)
&= \underbrace{(2\pi)^{-N}
   \int\limits_{\R^N} 
   \mathrm{e}^{\,\mathrm{i}\, x\;\cdot\;\xi}\,
   E_{\alpha}\!\big(-a(\xi)\, t^{\alpha}\big)\,
   \mathrm{d}\xi}_{\bz_1(t,x)} \\[0.5em]
&\quad + 
   \underbrace{\sigma(2\pi)^{-N}
   \int\limits_{0}^{t}\,\int\limits_{\R^N}\,\int\limits_{\R^N}
   \mathrm{e}^{\,\mathrm{i}\,(x-y)\;\cdot\;\xi}\,
   \Lambda(t-\tau,\xi)\, \bw(\tau,y)\,
   \mathrm{d}\xi\,\mathrm{d}y\,\mathrm{d}\tau}_{\bz_2(t,x)},
\end{aligned}
\end{equation}
where
\begin{equation}
    \label{Z1}
    a(\xi)=\lambda|\xi|^2+\mu \left(1-\hat{\bj}(\xi)\right),
\end{equation}
and
\begin{equation}
    \label{Z2}
\Lambda(t,\xi)= t^{\alpha-1} E_{\alpha, \alpha}\left(-t^\alpha a(\xi)\right).
    \end{equation}
\end{theorem}

\begin{definition}
\label{Mild}
The solution $\bz(t,x)$ is called {\it mild} if $\mathbb{E}[\bz^2(t,x)]<\infty$ for all $t>0, x\in\R^N$.
\end{definition}
Note that, in the purely local case $\mu=0$ in \eqref{Main-eq}, the existence of mild solutions in the sense of
Definition~\ref{Mild} was previously investigated in \cite[Theorem~3, p.~525]{Hach}, where only partial results were obtained.
Our second main result provides a complete classification of mild solutions, thereby sharpening and extending
\cite[Theorem~3, p.~525]{Hach} to the more general setting that includes the nonlocal operator
$\mu\big(\mathbf{J}\ast\bz-\bz\big)$.

In contrast to the purely local framework, the presence of the nonlocal term
$\mu\big(\mathbf{J}\ast\bz-\bz\big)$ modifies the Fourier symbol of the generator from $\lambda|\xi|^{2}$ to
\[
a(\xi)=\lambda|\xi|^{2}+\mu\bigl(1-\widehat \bj(\xi)\bigr),
\]
and the question of mildness becomes a genuinely mixed local--nonlocal space--time integrability problem.
Our analysis reveals, in particular, that in the purely nonlocal regime $\lambda=0$ and $\mu>0$, no mild solution
exists for any $\alpha\in(0,2)$ and any spatial dimension $N\ge1$.
On the other hand, when $\lambda>0$, the Laplacian term dominates the high-frequency behavior of the symbol,
and a sharp classification in terms of $(\alpha,N)$ is recovered.

More precisely, we have the following result.
\begin{theorem}
\label{Thm2}
Let $\bz(t,x)$ be the solution of \eqref{Main-eq} obtained in Theorem~\ref{Thm1}.  
Then the following assertions hold:
\begin{enumerate}[label=(\roman*)]
    \item If $\lambda = 0$ and $\mu>0$, the solution $\bz(t,x)$ is not mild for any spatial dimension $N \ge 1$ and for all $0 < \alpha < 2$.
    
    \item If $\lambda > 0$ and $\mu\geq 0$, then:
    \begin{itemize}
        \item[(a)] For $\alpha = 1$, the solution $\bz(t,x)$ is mild if and only if $N = 1$.
        \item[(b)] For $0 < \alpha < 1$, the solution $\bz(t,x)$ is  mild if and only if  $N=1$ and $2/3<\alpha<1$.
        \item[(c)] For $1 < \alpha < 2$, the solution $\bz(t,x)$ is mild if and only if $N \in \{1,2\}$.
    \end{itemize}
\end{enumerate}
\end{theorem}

\begin{remark}
\rm
~\begin{enumerate}[label=(\roman*)]
\item It is worth noting that the particular case $\mu=0$ with $\lambda>0$ has already been investigated in \cite{Hach}.

\item When $\lambda>0$ and $\mu=0$, the results obtained here improve those of \cite[Theorem~3, p.~525]{Hach}. More precisely, \cite[Theorem~3(c), p.~525]{Hach} shows that the solution fails to be mild in any space dimension for the entire range $\alpha\in(0,1)$. In contrast, we prove that mild solutions do exist in dimension $N=1$, provided the additional condition $2/3<\alpha<1$ is satisfied.

\item From an asymptotic viewpoint, the local diffusion operator $\lambda\Delta$ dominates the nonlocal contribution $\mu\big(\mathbf{J}\ast I-1\big)$, so that the latter becomes negligible in the long-time regime.

\item Similar conclusions remain valid for \eqref{Main-eq} when the nonlocal term $\big(\mathbf{J}\ast I-1\big)$ is replaced by its higher-order analogue $(-1)^{\,n-1}\big(\mathbf{J}\ast I-1\big)^{n}$, for any integer $n\ge1$.

\end{enumerate}
\end{remark}
The paper is organized as follows.
Section~\ref{Prelim} collects the analytical and probabilistic tools needed for our analysis, including basic facts on tempered distributions, Mittag--Leffler functions, Caputo fractional derivatives, and time--space white noise.
In Section~\ref{S3}, we prove the main results of the paper: we first derive an explicit representation of the solution and then establish sharp criteria for the existence of mild solutions.
Finally, Section~\ref{sec:num} is devoted to numerical simulations that illustrate the behavior of the solution in several representative regimes, with particular emphasis on the evolution of the mean and variance, as well as on subdiffusive effects and memory phenomena induced by fractional time derivatives.\\

In the remainder of the paper, the notation $\asymp$ denotes two--sided estimates, namely that the corresponding ratio is bounded from above and below by positive constants. We write $X\lesssim Y$ (equivalently, $Y\gtrsim X$) to indicate that $X\leq C\,Y$ for some constant $C>0$ independent of the relevant parameters. Finally, the notation
\[
f(z)\underset{z\to z_0}{\sim} g(z)
\]
means that $\displaystyle \lim_{z\to z_0}\frac{f(z)}{g(z)}=1$, where $z_0\in[-\infty,\infty]$.

\section{Preliminaries}\label{Prelim}

\subsection{The space of tempered distributions}
For the reader's convenience, we briefly recall some basic facts about the Schwartz space
$\mathcal S=\mathcal S(\R^N)$ of rapidly decreasing smooth functions and its dual
$\mathcal S'=\mathcal S'(\R^N)$, the space of tempered distributions.

Let $\mathcal S(\R^N)$ be the space of all real-valued $C^\infty$ functions $f$ on $\R^N$ such that
$f$ and all its derivatives decay faster than any polynomial at infinity.
The topology of $\mathcal S(\R^N)$ is generated by the seminorms
\[
\|f\|_{k,\alpha}
:=\sup_{y\in\R^N}(1+|y|^k)\,|\partial^\alpha f(y)|<\infty,
\]
where $k\in\N$ and $\alpha=(\alpha_1,\dots,\alpha_N)\in\N^N$ is a multi-index, with
\[
\partial^\alpha f
:=\frac{\partial^{|\alpha|}}{\partial y_1^{\alpha_1}\cdots \partial y_N^{\alpha_N}}f,
\qquad
|\alpha|:=\sum_{j=1}^N\alpha_j.
\]
With this structure, $\mathcal S(\R^N)$ is a Fr\'echet space.

Its (topological) dual $\mathcal S'(\R^N)$ consists of all continuous linear functionals on
$\mathcal S(\R^N)$ (with the weak-$\ast$ topology) and is called the space of
\emph{tempered distributions}. For $\Phi\in\mathcal S'$ and $f\in\mathcal S$, we write
\[
\Phi(f)\qquad\text{or}\qquad \langle \Phi,f\rangle
\]
for the canonical pairing.

\subsection*{Fourier transform and Plancherel theorem}
For $f\in L^1(\R^N)$, we define the Fourier transform by
\[
\widehat f(\xi):=\int_{\R^N}e^{-ix\cdot\xi}\,f(x)\,dx,
\qquad \xi\in\R^N,
\]
and the inverse transform by
\[
f(x)=(2\pi)^{-N}\int_{\R^N}e^{ix\cdot\xi}\,\widehat f(\xi)\,d\xi,
\qquad x\in\R^N,
\]
whenever both sides are well defined. The Fourier transform extends uniquely to a unitary
operator on $L^2(\R^N)$ and satisfies the Plancherel identity
\begin{equation}\label{planch}
\int_{\R^N}|f(x)|^2\,dx
=(2\pi)^{-N}\int_{\R^N}|\widehat f(\xi)|^2\,d\xi,
\qquad f\in L^2(\R^N).
\end{equation}
Moreover, for sufficiently regular functions,
\[
\widehat{\partial_{x_j}f}(\xi)=i\xi_j\widehat f(\xi),
\qquad
\widehat{(-\Delta f)}(\xi)=|\xi|^2\widehat f(\xi),
\]
and these identities will be used repeatedly.

\subsection{The Mittag--Leffler functions}
\begin{definition}
Let $\alpha,\beta>0$. The two-parameter Mittag--Leffler function is defined by
\begin{equation}\label{Mitag-two}
E_{\alpha,\beta}(z):=\sum_{k=0}^\infty \frac{z^k}{\Gamma(\alpha k+\beta)},
\qquad z\in\C.
\end{equation}
In the special case $\beta=1$, we write
\begin{equation}\label{Mitag}
E_\alpha(z):=E_{\alpha,1}(z)=\sum_{k=0}^\infty\frac{z^k}{\Gamma(\alpha k+1)}.
\end{equation}
\end{definition}
Note that $E_1(z)=e^z$. Define
\begin{equation}\label{c-alpha}
\kappa_\alpha:=\frac{\sin(\pi\alpha)}{\pi}\,\Gamma(1+\alpha),
\end{equation}
which is positive for $0<\alpha<1$ and negative for $1<\alpha<2$, and set
\begin{equation}\label{e-alpha}
e_\alpha(z):=z\,e^{z^{1/\alpha}}-\kappa_\alpha.
\end{equation}
Then one has the classical expansion
\begin{equation}\label{Asympt-E-alpha}
\alpha z E_\alpha(z)=e_\alpha(z)+O\!\left(\frac1z\right),\qquad |z|\to\infty.
\end{equation}

We also recall the large-argument asymptotics of $E_\alpha(-x)$ and $E_{\alpha,\alpha}(-x)$.
\begin{theorem}\label{thm:asympt-ML}
Let $\alpha\in(0,2)$. As $x\to+\infty$, the following asymptotics hold:
\begin{enumerate}
\item If $0<\alpha<1$, then
\begin{align}
E_\alpha(-x)&\sim \frac{x^{-1}}{\Gamma(1-\alpha)}, \label{E-alpha-1}\\
E_{\alpha,\alpha}(-x)&\sim \frac{\pi\,x^{-2}}{\sin(\pi\alpha)\Gamma(1+\alpha)}. \label{E-alpha-alpha-1}
\end{align}
\item If $1\le\alpha<2$, then
\begin{align}
|E_\alpha(-x)|
&\sim\frac1\alpha\exp\!\Big(x^{1/\alpha}\cos(\tfrac\pi\alpha)\Big), \label{E-alpha-2}\\
|E_{\alpha,\alpha}(-x)|
&\sim\frac{x^{(1-\alpha)/\alpha}}{\alpha}\exp\!\Big(x^{1/\alpha}\cos(\tfrac\pi\alpha)\Big). \label{E-alpha-alpha-2}
\end{align}
\end{enumerate}
\end{theorem}

\begin{remark}
\rm The proof of Theorem~\ref{thm:asympt-ML} can be found for instance in
\cite[Proposition~3.6 and Theorem~4.3]{Gorenflo2020}; see also \cite{FCAA-2002,Exo-Book}.
\end{remark}

\begin{remark}
\rm For $\alpha=\tfrac12$, one has the representation (see, e.g., \cite[p.~107]{Exo-Book})
\[
E_{1/2}(-x)=e^{x^2}\,\erfc(x),
\qquad
\erfc(z):=\frac{2}{\sqrt{\pi}}\int_z^\infty e^{-t^2}\,dt.
\]
\end{remark}

\begin{proposition}\label{prop:sharp-asympt-ML}
The following sharp bounds hold for real-valued Mittag--Leffler functions:
\begin{align}
\label{sharp-ML1}
\frac{1}{1+\Gamma(1-\alpha)x}
&\leq
E_{\alpha}(-x)
\leq
\frac{1}{1+\dfrac{x}{\Gamma(1+\alpha)}},
\qquad x \geq 0,\; 0 < \alpha < 1, \\[1ex]
\label{sharp-ML2}
\frac{1}{\Bigl(1+\sqrt{\dfrac{\Gamma(1-\alpha)}{\Gamma(1+\alpha)}}\,x\Bigr)^2}
&\leq
\Gamma(\alpha)E_{\alpha,\alpha}(-x)
\leq
\frac{1}{\Bigl(1+\sqrt{\dfrac{\Gamma(1+\alpha)}{\Gamma(1+2\alpha)}}\,x\Bigr)^2},
\qquad x \geq 0,\; 0 < \alpha < 1, \\[1ex]
\label{sharp-ML3}
\frac{1}{1+\dfrac{\Gamma(\beta-\alpha)}{\Gamma(\beta)}x}
&\leq
\Gamma(\beta)E_{\alpha,\beta}(-x)
\leq
\frac{1}{1+\dfrac{\Gamma(\beta)}{\Gamma(\beta+\alpha)}x},
\qquad x \geq 0,\; 0 < \alpha \leq 1,\; \beta > \alpha .
\end{align}
\end{proposition}

\noindent
For \eqref{sharp-ML1} see \cite[Theorem~4]{Simon2014}, while \eqref{sharp-ML2}--\eqref{sharp-ML3}
follow from \cite[Proposition~4]{Boudabsa2021}.

\medskip

 The distribution of zeros of the Mittag--Leffler function \(E_{\alpha}\) is a delicate problem that has been investigated extensively in the literature. Below we summarize the structure of the real zeros. A detailed analysis, together with further properties of \(E_{\alpha}\), may be found in \cite{Gorenflo2020,MLF,Pollard,Popov}.
 \begin{theorem}\label{thm:zeros_ML}
Let $\alpha>0$ and define
\begin{equation}
    \label{Z-alpha}
    \mathcal{Z}_\alpha=\Big\{x\in\mathbb{R}:\ E_\alpha(x)=0\Big\},
\end{equation}
the set of real zeros of the Mittag--Leffler function $E_\alpha$. Then the following properties hold:
\begin{enumerate}[label=(\roman*)]
    \item For every $\alpha>0$, one has $E_\alpha(x)>0$ for all $x>0$; hence 
    \[
    \mathcal{Z}_\alpha\subset(-\infty,0].
    \]
    \item If $0<\alpha\le1$, the function $E_\alpha(-x)$ is completely monotone on $(0,\infty)$ (Pollard’s theorem).  
    Consequently, $E_\alpha(-x)>0$ for all $x>0$, and therefore 
    \[
    \mathcal{Z}_\alpha=\varnothing.
    \]
    \item When $1<\alpha<2$, the function $E_\alpha$ admits a finite number of simple real zeros, all located in $(-\infty,0)$.  
    Moreover, the number of these zeros increases as $\alpha\uparrow2$.
    \item For $\alpha=2$, one has the explicit representation $E_2(x)=\cosh(\sqrt{x})$, which implies
    \[
    Z_2=\Big\{-\big(\tfrac{\pi}{2}+n\pi\big)^2:\ n=0,1,2,\dots\Big\}.
    \]
    \item For $\alpha>2$, the pattern of real zeros becomes irregular; no general closed form is known.  
    The existence or absence of real zeros depends sensitively on the particular value of $\alpha$.
\end{enumerate}
\end{theorem}

\medskip

From Theorem~\ref{thm:asympt-ML}, we obtain the following consequence.

\begin{lemma}\label{lem:ML-Lp}
Let $\alpha\in(0,2)$ and $1\le p\le\infty$. Let $\mathbf b:\R^N\to[0,\infty)$ be continuous and assume
\begin{equation}\label{b-y}
\mathbf b(y)\sim \lambda |y|^2\qquad \text{as }|y|\to\infty
\end{equation}
for some $\lambda>0$. Then:
\begin{enumerate}
\item If $0<\alpha<1$, then
\begin{equation}\label{E-alpha-Lp-1}
\begin{aligned}
y\mapsto E_\alpha(-\mathbf b(y))\in L^p(\R^N)
&\Longleftrightarrow N<2p,\\[1mm]
y\mapsto E_{\alpha,\alpha}(-\mathbf b(y))\in L^p(\R^N)
&\Longleftrightarrow N<4p.
\end{aligned}
\end{equation}
\item If $1\le\alpha<2$, then both maps
$y\mapsto E_\alpha(-\mathbf b(y))$ and $y\mapsto E_{\alpha,\alpha}(-\mathbf b(y))$
belong to $L^p(\R^N)$.
\end{enumerate}
\end{lemma}

We will also use the following sharp time--space integrability criterion. 
\begin{lemma}\label{lem:b}
Let $0<\alpha<1$ and let $\mathbf b:\R^N\to[0,\infty)$ be continuous such that
$\mathbf b(0)=0$ and $\mathbf b(y)\sim \lambda|y|^2$ as $|y|\to\infty$ for some $\lambda>0$.
Then
\[
\int_0^1\int_{\R^N}\frac{s^{2\alpha-2}}{(1+s^\alpha \mathbf b(y))^4}\,dy\,ds<\infty
\quad\Longleftrightarrow\quad
N=1\ \text{ and }\ \alpha>\frac23.
\]
\end{lemma}
The proof of Lemma~\ref{lem:b} is routine and follows from standard estimates; the details are omitted.
\begin{proposition}\label{prop:integrability-Eaa}
Let $\alpha\in(1,2)$ and let $E_{\alpha,\alpha}$ denote the Mittag--Leffler function.
Assume that $\mathbf b:\R^N\to[0,\infty)$ is continuous and satisfies
\begin{equation}\label{eq:b-asympt}
\mathbf b(y)\sim \lambda |y|^2 \qquad \text{as } |y|\to\infty
\end{equation}
for some $\lambda>0$. Then, for every $t>0$,
\[
\int_0^t\int_{\R^N} s^{2\alpha-2}\Bigl(E_{\alpha,\alpha}(-s^\alpha \mathbf b(y))\Bigr)^2\,dy\,ds<\infty
\quad\Longleftrightarrow\quad
N<4-\frac{2}{\alpha}.
\]
\end{proposition}
\begin{proof}
Let $\alpha\in(1,2)$ and set $c_\alpha=-\cos(\pi/\alpha)>0$. The classical asymptotics of the Mittag--Leffler function give
\[
|E_{\alpha,\alpha}(-x)|\asymp x^{\frac{1-\alpha}{\alpha}}e^{-c_\alpha x^{1/\alpha}},
\qquad x\to+\infty,
\]
while $E_{\alpha,\alpha}$ is bounded on bounded intervals. Since $\mathbf{b}(y)\sim\lambda|y|^2$, there exist $R_0>0$ and $c_1,c_2>0$ such that
$c_1|y|^2\le \mathbf{b}(y)\le c_2|y|^2$ for $|y|\ge R_0$. The contribution of $\{|y|\le R_0\}$ is finite because $\alpha>1$.

Fix $s\in(0,t]$. On the region $\{s^\alpha \mathbf{b}(y)\le1\}$, boundedness of $E_{\alpha,\alpha}$ and the quadratic growth of $b$ yield
\[
\int_{\{s^\alpha \mathbf{b}(y)\le1\}} s^{2\alpha-2}\bigl(E_{\alpha,\alpha}(-s^\alpha \mathbf{b}(y))\bigr)^2dy
\lesssim s^{2\alpha-2-\frac{\alpha N}{2}}.
\]
On the complementary region $\{s^\alpha \mathbf{b}(y)\ge1\}$, the asymptotic formula gives
\[
s^{2\alpha-2}\bigl(E_{\alpha,\alpha}(-s^\alpha \mathbf{b}(y))\bigr)^2
\asymp |y|^{-\frac{4(\alpha-1)}{\alpha}}
e^{-\kappa s|y|^{2/\alpha}},
\]
for some $\kappa>0$. A change to polar coordinates and the substitution
$v=s r^{2/\alpha}$ show that
\[
\int_{\{s^\alpha \mathbf{b}(y)\ge1\}} s^{2\alpha-2}
\bigl(E_{\alpha,\alpha}(-s^\alpha \mathbf{b}(y))\bigr)^2dy
\asymp s^{-\left(\frac{\alpha}{2}(N-4)+2\right)}.
\]
Therefore,
\[
\int_{\R^N} s^{2\alpha-2}\bigl(E_{\alpha,\alpha}(-s^\alpha \mathbf{b}(y))\bigr)^2dy
\lesssim
s^{2\alpha-2-\frac{\alpha N}{2}}
+
s^{-\left(\frac{\alpha}{2}(N-4)+2\right)},
\]
and both terms are integrable near $s=0$ if and only if $N<4-\frac{2}{\alpha}$.

Conversely, if $N\ge 4-\frac{2}{\alpha}$, one can restrict the spatial integral to an annulus of radius comparable to $s^{-\alpha/2}$, where $s^\alpha \mathbf{b}(y)\gtrsim1$ and the lower asymptotic bound applies. This yields
\[
\int_{\R^N} s^{2\alpha-2}\bigl(E_{\alpha,\alpha}(-s^\alpha \mathbf{b}(y))\bigr)^2dy
\gtrsim s^{-\left(\frac{\alpha}{2}(N-4)+2\right)},
\]
which is not integrable at $s=0$. The claim follows.
\end{proof}

\subsection{The Caputo fractional derivative}
We recall the Caputo (Abel--Caputo) fractional derivative and some Laplace transform identities.
For general references on fractional calculus, see \cite{KD,KST,Podlubny1999,SKM}.

\begin{definition}
Let $\alpha>0$ and $n=\lceil\alpha\rceil$. The Caputo fractional derivative of a function $f=f(t)$
is defined by
\[
\partial_t^\alpha f(t):=
\begin{cases}
\displaystyle\frac{1}{\Gamma(n-\alpha)}\int_0^t (t-s)^{n-1-\alpha} f^{(n)}(s)\,ds,
& n-1<\alpha<n,\\[2mm]
\displaystyle\frac{d^n}{dt^n}f(t),& \alpha=n.
\end{cases}
\]
If $f$ is not sufficiently smooth, $\,\partial_t^\alpha f\,$ is understood in the sense of distributions.
\end{definition}

\subsubsection{Laplace transform of Caputo derivatives}
For $s>0$, the Laplace transform satisfies
\begin{align}
\mathcal{L}\!\left[\partial^\alpha_t f\right](s)
&= s^{\alpha}(\mathcal{L}f)(s) - s^{\alpha-1}f(0),
&& 0<\alpha \leq 1, \label{eq:laplace1}\\
\mathcal{L}\!\left[\partial^\alpha_t f\right](s)
&= s^{\alpha}(\mathcal{L}f)(s) - s^{\alpha-1}f(0)-s^{\alpha-2}f'(0),
&& 1<\alpha \leq 2, \label{eq:laplace11}\\
\mathcal{L}\!\left[E_{\alpha}(b t^{\alpha})\right](s)
&= \frac{s^{\alpha-1}}{s^{\alpha} - b},
&& \alpha>0,\ b\in\R,\ s^\alpha>|b|, \label{eq:laplace2}\\
\mathcal{L}\!\left[t^{\alpha-1} E_{\alpha,\alpha}(-b t^{\alpha})\right](s)
&= \frac{1}{s^{\alpha} + b},
&& \alpha>0,\ b\in\R,\ s^\alpha>|b|. \label{eq:laplace3}
\end{align}

\subsection{Time--space white noise}
Let $\Omega:=\mathcal{S}'(\R^N)$ be equipped with the weak-$\ast$ topology and let
$\mathcal{F}:=\mathcal{B}(\mathcal{S}'(\R^N))$ be the Borel $\sigma$-algebra.
The underlying Gaussian (white-noise) probability measure is given by the Bochner--Minlos theorem.

\begin{theorem}[Bochner--Minlos]
There exists a unique probability measure $\mathbb{P}$ on $(\Omega,\mathcal{F})$ such that
\[
\mathbb{E}\!\left[e^{i\langle\cdot,\phi\rangle}\right]
=\int_{\mathcal{S}'} e^{i\langle\omega,\phi\rangle}\,d\mathbb{P}(\omega)
= \exp\!\left(-\frac12\|\phi\|_{L^2(\R^N)}^2\right),
\qquad \forall\,\phi\in\mathcal{S}(\R^N).
\]
\end{theorem}

It follows that $\mathbb{E}[\langle\omega,\phi\rangle]=0$ and
$\mathbb{E}[\langle\omega,\phi\rangle^2]=\|\phi\|_{L^2(\R^N)}^2$ (It\^o isometry).
Using this isometry, the pairing extends to $\phi\in L^2(\R^N)$ by approximation.

For $x=(x_1,\dots,x_N)\in\R^N$, set
\[
\widetilde{B}(x):=\big\langle \omega,\chi_{[0,x_1]\times\cdots\times[0,x_N]}\big\rangle,
\]
where $[0,x_i]$ is understood as $[x_i,0]$ if $x_i<0$.
This process admits a continuous modification $\bb(x,\omega)$, defining an $N$-parameter Brownian motion
(Brownian sheet for $N\ge2$).
For $\phi\in L^2(\R^N)$, the $N$-parameter Wiener--It\^o integral is
\[
\int_{\R^N}\phi(x)\,d\bb(x,\omega):=\langle\omega,\phi\rangle.
\]

In the sequel we write $\bb(t,x):=\bb(t,x,\omega)$, $t\ge0$, $x\in\R^N$,
for the time--space Brownian motion (Brownian sheet).
Since $(t,x)\mapsto\bb(t,x)$ is almost surely continuous, we may define its distributional derivatives, and in particular
the time--space white noise by
\begin{equation}\label{white-noise}
\bw(t,x):=\frac{\partial}{\partial t}\frac{\partial^N}{\partial x_1\cdots\partial x_N}\bb(t,x).
\end{equation}
Equivalently, $\bw$ may be interpreted as an element of the Hida space $(\mathcal S)^\ast$ of stochastic distributions.
In this setting, the It\^o--Skorohod integral with respect to $\bb(dt,dx)$ can be written via Wick product $\circ$ as (see, e.g., \cite{LindstromUboe,Benth})
\begin{equation}\label{Ito-Skorohod}
\int_0^T\int_{\R^N} f(t,x,\omega)\,\bb(dt,dx)
=\int_0^T\int_{\R^N} f(t,x,\omega)\circ \bw(t,x)\,dt\,dx.
\end{equation}
In particular, if $f$ is deterministic, then
\begin{equation}\label{Wiener-integral}
\int_0^T\int_{\R^N} f(t,x)\,\bb(dt,dx)
=\int_0^T\int_{\R^N} f(t,x)\,\bw(t,x)\,dt\,dx.
\end{equation}
This is the stochastic integration interpretation adopted throughout the paper.

\section{Proof of  main results}
\label{S3}

\subsection{Proof of Theorem \ref{Thm1}}
Assume that $\bz(t, x)$ is a solution of \eqref{Main-eq}-\eqref{eq:IC-Dirac}.  
Applying the Laplace transform $\mathcal{L}$ to both sides of \eqref{Main-eq}, we obtain
\[
s^{\alpha} \tilde{\bz}(s, x) - s^{\alpha-1} \bz(0, x) 
= \bigl[\bj \ast \tilde{\bz}\bigr](s, x) - \tilde{\bz}(s, x) + \sigma \tilde{W}(s, x). 
\]
Next, applying the Fourier transform $\mathcal{F}$ and using $\hat{\bz}(0, \xi) = 1$, we get
\[
s^{\alpha} \hat{\tilde{\bz}}(s, \xi) - s^{\alpha-1} 
= -a(\xi)\,\hat{\tilde{\bz}}(s, \xi) + \sigma \hat{\tilde{W}}(s, \xi),
\]
where $a(\xi)$ is given by \eqref{Z1}.  
Thus,
\begin{equation}\label{F-L}
\bigl(s^{\alpha} + a(\xi)\bigr) \hat{\tilde{\bz}}(s, \xi) 
= s^{\alpha-1} + \sigma \hat{\tilde{W}}(s, \xi).
\end{equation}

Since $|\hat{\bj}(\xi)| \leq \int \bj(x)\,dx = 1$ and $\lambda, \mu\ge 0$, it follows that $a(\xi) \geq 0$.  
Therefore, from \eqref{F-L} we deduce
\begin{equation}\label{F-L1}
\hat{\tilde{\bz}}(s, \xi) 
= \frac{s^{\alpha-1}}{s^{\alpha} + a(\xi)} 
+ \frac{\sigma \hat{\tilde{W}}(s, \xi)}{s^{\alpha} + a(\xi)}.
\end{equation}
Since the Laplace and Fourier transforms commute, we can rewrite this as
\begin{equation}\label{F-L2}
\tilde{\hat{\bz}}(s, \xi) 
= \frac{s^{\alpha-1}}{s^{\alpha} + a(\xi)} 
+ \frac{\sigma \tilde{\hat{W}}(s, \xi)}{s^{\alpha} + a(\xi)}.
\end{equation}

Applying the inverse Laplace transform $\mathcal{L}^{-1}$ to \eqref{F-L2}, we obtain
\begin{equation}\label{Four-Z1}
\begin{split}
\hat{\bz}(t, \xi) 
&= \mathcal{L}^{-1}\!\left( \frac{s^{\alpha-1}}{s^{\alpha} + a(\xi)} \right)(t, \xi) 
+ \mathcal{L}^{-1}\!\left( \frac{\sigma \tilde{\hat{W}}(s, \xi)}{s^{\alpha} + a(\xi)} \right)(t, \xi) \\
&= E_{\alpha}\!\bigl(-a(\xi) t^{\alpha}\bigr) 
+ \mathcal{L}^{-1}\!\left( \frac{\sigma \tilde{\hat{W}}(s, \xi)}{s^{\alpha} + a(\xi)} \right)(t, \xi),
\end{split}
\end{equation}
where $E_{\alpha}$ is the Mittag--Leffler function given by \eqref{Mitag}.
It remains to compute 
$$\mathcal{L}^{-1}\!\left(\frac{\sigma \tilde{\hat{W}}(s, \xi)}{s^{\alpha} + a(\xi)} \right)(t, \xi).$$
Note that
\begin{equation}\label{inv-Lap1}
\frac{\sigma}{s^{\alpha} + a(\xi)} = \sigma \,\mathcal{L}\bigl(\Lambda(t,\xi)\bigr),
\end{equation}
where $\Lambda(t,\xi)$ is given by \eqref{Z2}. Hence,
\begin{equation}\label{inv-Lap2}
\begin{split}
\mathcal{L}^{-1}\!\left(\frac{\sigma \tilde{\hat{W}}(s, \xi)}{s^{\alpha} + a(\xi)} \right)(t, \xi) 
&= \sigma\, \mathcal{L}^{-1}\!\left(\mathcal{L}[\hat{W}(t,\xi)](s)\,\mathcal{L}[\Lambda(t,\xi)](s)\right) \\
&= \sigma \int_0^t \Lambda(t-\tau,\xi)\,\hat{W}(\tau,\xi)\,d\tau.
\end{split}
\end{equation}

Substituting \eqref{inv-Lap2} into \eqref{Four-Z1}, we obtain
\begin{equation}\label{Four-Z2}
\hat{\bz}(t, \xi) 
= E_{\alpha}\!\bigl(-a(\xi) t^{\alpha}\bigr) 
+ \sigma \int_0^t \Lambda(t-\tau,\xi)\,\hat{W}(\tau,\xi)\,d\tau.
\end{equation}

Finally, applying the inverse Fourier transform gives
\begin{equation}\label{Inv-Four}
\begin{split}
\bz(t,x) 
&= (2\pi)^{-N} \int_{\R^N} e^{i x\cdot \xi}\, E_{\alpha}\!\bigl(-a(\xi) t^\alpha\bigr)\, d\xi \\
&\quad + \sigma \int_0^t \mathcal{F}^{-1}\!\left( \Lambda(t-\tau, \xi)\, \hat{W}(\tau,\xi) \right) d\tau \\
&= (2\pi)^{-N} \int_{\R^N} e^{i x\cdot \xi}\, E_{\alpha}\!\bigl(-a(\xi) t^\alpha\bigr)\, d\xi \\
&\quad + \sigma (2\pi)^{-N} \int_0^t \int_{\R^N}\int_{\R^N} e^{i(x-y)\cdot \xi}\, \Lambda(t-\tau, \xi)\, \bw(\tau, y)\, d\xi\,dy\,d\tau.
\end{split}
\end{equation}
This concludes the proof of Theorem~\ref{Thm1}.

\subsection{Proof of Theorem \ref{Thm2}} 
Recall that according to Theorem~\ref{Thm1}, the solution $\bz$ can be expressed as 
\begin{equation}
    \label{ZZ}
    \bz(t,x)=\bz_1(t,x)+\bz_2(t,x),
\end{equation}
where 
\begin{equation}
    \label{z12}
    \begin{split}
        \bz_1(t,x)&= (2\pi)^{-N}
   \int\limits_{\R^N} 
   \mathrm{e}^{\,\mathrm{i}\, x\;\cdot\;\xi}\,
   E_{\alpha}\!\big(-a(\xi)\, t^{\alpha}\big)\,
   \mathrm{d}\xi,\\
        \bz_2(t,x)&=\sigma(2\pi)^{-N}
   \int\limits_{0}^{t}(t-\tau)^{\alpha-1}\,\int\limits_{\R^N}\,\left(\int\limits_{\R^N}
   \mathrm{e}^{\,\mathrm{i}\,(x-y)\;\cdot\;\xi}\,
   E_{\alpha, \alpha}\left(-(t-\tau)^\alpha a(\xi)\right)d\xi\right)\,
   \bb(dy, d\tau),
    \end{split}
\end{equation}
and $a(\xi)$ is given by \eqref{Z1}. To compute $\mathbb{E}\left[\bz^2(t,x)\right]$, we use It\^o isometry to get
\begin{equation}
    \label{EST-1}
  \mathbb{E}\left[\bz^2(t,x)\right]=\mathcal{M}_1(t,x)+ \mathcal{M}_2(t,x), 
\end{equation}
where
\begin{equation}
    \label{Est-2}
    \begin{split}
        \mathcal{M}_1(t,x)&=(2\pi)^{-2N}\left(
   \int\limits_{\R^N} 
   \mathrm{e}^{\,\mathrm{i}\, x\;\cdot\;\xi}\,
   E_{\alpha}\!\big(-a(\xi)\, t^{\alpha}\big)\,
   \mathrm{d}\xi\right)^2,\\
        \mathcal{M}_2(t,x)&=\sigma^{2} (2\pi)^{-2N} 
\int\limits_{0}^{t} \int\limits_{\mathbb{R}^{N}} 
(t - \tau)^{2\alpha - 2} 
\left( 
    \int\limits_{\mathbb{R}^{N}} 
        e^{i(x - y)\;\cdot\;\xi}\;
        E_{\alpha, \alpha}\!\left(-(t - \tau)^{\alpha} a(\xi)\right)
    d\xi
\right)^{2} 
dy\,d\tau.
    \end{split}
\end{equation}
By Plancherel formula \eqref{planch}, we can rewrite $\mathcal{M}_2(t,x)$ as
\begin{equation}
    \label{M-2-planch}
    \begin{split}
    \mathcal{M}_2(t,x)&=\sigma^2\,(2\pi)^{-N}\int\limits_0^t\,\int\limits_{\R^N}\,(t - \tau)^{2\alpha - 2}\, \bigg( E_{\alpha, \alpha} \big(-(t - \tau)^{\alpha} a(x-y)\big)\bigg)^2\,dy\,d\tau\\
    &=\sigma^2\,(2\pi)^{-N}\int\limits_0^t\,\int\limits_{\R^N}\,s^{2\alpha - 2}\,\bigg( E_{\alpha, \alpha} \big(-s^{\alpha} a(\xi)\big)\bigg)^2\,d\xi\,ds.
    \end{split}
\end{equation}
We first examine the case $\lambda=0$. Since $a(\xi)\to\mu>0$ as $|\xi|\to\infty$, it follows that
\begin{equation}
\label{lambda0}
\bigl|E_{\alpha}(-a(\xi)t^{\alpha})\bigr|
\;\xrightarrow[|\xi|\to\infty]{}\;
\bigl|E_{\alpha}(-\mu\,t^{\alpha})\bigr|.
\end{equation}

If $0<\alpha\le 1$, Theorem~\ref{thm:zeros_ML}\,(ii) ensures that
$E_{\alpha}(-\mu\,t^{\alpha})>0$ for all $t>0$. Hence, for every $t>0$, the mapping
$\xi\mapsto E_{\alpha}(-a(\xi)t^{\alpha})$ does not belong to $L^{1}(\mathbb{R}^{N})$
for any $N\ge1$.

If $1<\alpha<2$, Theorem~\ref{thm:zeros_ML}\,(iii) yields
$\bigl|E_{\alpha}(-\mu\,t^{\alpha})\bigr|>0$ for almost every $t>0$. As a consequence,
for almost every $t>0$, the function
$\xi\mapsto E_{\alpha}(-a(\xi)t^{\alpha})$ fails to be integrable over $\mathbb{R}^{N}$.

We therefore conclude that the solution cannot be mild for any $\alpha\in(0,2)$ and
any space dimension $N\ge1$. This completes the proof of
Theorem~\ref{Thm2}\,(i).\\
We next turn to the case $\lambda>0$.  
Let $\alpha\in(0,1)$. By Lemma~\ref{lem:ML-Lp}, the mapping
\[
\xi \longmapsto E_{\alpha}\!\left(-a(\xi)t^{\alpha}\right)
\]
belongs to $L^{1}(\mathbb{R}^{N})$ if and only if $N=1$. Consequently,
\[
\mathcal{M}_{1}(t,x)<\infty \quad \Longleftrightarrow \quad N=1.
\]
Furthermore, combining \eqref{M-2-planch} with Lemma~\ref{lem:b}, we obtain
\[
\mathcal{M}_{2}(t,x)<\infty 
\quad \Longleftrightarrow \quad N=1 \ \text{and}\ \frac{2}{3}<\alpha<1.
\]
Hence, in the range $\alpha\in(0,1)$, the solution is mild if and only if
$N=1$ and $2/3<\alpha<1$.

We now consider $\alpha\in(1,2)$. By Theorem~\ref{thm:asympt-ML}\,(2), we have
\[
\mathcal{M}_{1}(t,x)<\infty \qquad \text{for all } N\ge1.
\]
Applying again Theorem~\ref{thm:asympt-ML}\,(2) together with
\eqref{M-2-planch} and using the asymptotic behavior
$a(\xi)\sim \lambda|\xi|^{2}$ as $|\xi|\to\infty$, we infer that
\[
\mathcal{M}_{2}(t,x)<\infty 
\quad \Longleftrightarrow \quad (4-N)\alpha>2.
\]
Since $\alpha<2$, this condition implies $N<3$, namely $N=1,2$.
Moreover, for $N=1,2$, the inequality $(4-N)\alpha>2$ is automatically
satisfied because $\alpha>1$. Therefore, for $\alpha\in(1,2)$, the
solution is mild if and only if $N=1$ or $N=2$.

Finally, in the  case $\alpha=1$, the solution is mild only in
dimension $N=1$, since the condition $(4-N)\alpha>2$ under the restriction
$N=1,2$ forces $N=1$.

This completes the proof of Theorem~\ref{Thm2}.

\section{Numerical simulations}
\label{sec:num}

\noindent\underline{\sf Classical homogeneous heat equation.}
Consider the classical heat equation on $\R$,
\begin{equation}\label{eq:heat_classical}
\partial_t \bz(t,x)=\lambda\,\Delta\bz(t,x),
\qquad t>0,\; x\in\R,
\end{equation}
where $\lambda>0$ is the diffusion coefficient.
We impose the Dirac initial datum and vanishing condition at infinity,
\begin{equation}\label{eq:ICBC_classical}
\bz(0,x)=\delta_0(x),
\qquad
\lim_{|x|\to\infty}\bz(t,x)=0.
\end{equation}
The corresponding fundamental solution is the Gaussian heat kernel
\begin{equation}\label{eq:gaussian_kernel}
\bz(t,x)=\frac{1}{\sqrt{4\pi\lambda t}}
\exp\!\Bigl(-\frac{x^2}{4\lambda t}\Bigr),
\qquad t>0.
\end{equation}

\medskip
\noindent\underline{\sf Inhomogeneous heat equation.}
If a deterministic source term $g=g(t,x)$ is added,
\begin{equation}\label{eq:heat_inhom}
\partial_t \bz(t,x)=\lambda\,\Delta\bz(t,x)+g(t,x),
\qquad t>0,\; x\in\R,
\end{equation}
then the solution can be written in Duhamel form as
\begin{equation}\label{eq:duhamel_det}
\bz(t,x)=\bz_h(t,x)+\bz_p(t,x),
\end{equation}
where $\bz_h$ is the homogeneous solution \eqref{eq:gaussian_kernel} and
\begin{equation}\label{eq:duhamel_particular}
\bz_p(t,x)=\int_0^t\int_{\R}
\frac{1}{\sqrt{4\pi\lambda(t-\tau)}}
\exp\!\Bigl(-\frac{(x-y)^2}{4\lambda(t-\tau)}\Bigr)\,
g(\tau,y)\,dy\,d\tau,
\end{equation}
whenever the integrals are meaningful (e.g.\ under standard $L^2$-type assumptions on $g$).

\medskip
\noindent\underline{\sf Heat equation driven by time--space white noise.}
We now replace $g(t,x)$ by a scaled space--time white noise $\bw(t,x)$ and consider
\begin{equation}\label{eq:stoch_heat_alpha1}
\partial_t\bz(t,x)=\lambda\,\Delta\bz(t,x)+\sigma\,\bw(t,x),
\qquad \sigma\in\R.
\end{equation}
Formally, the mild solution reads
\begin{equation}\label{eq:mild_alpha1}
\bz(t,x)=\frac{1}{\sqrt{4\pi\lambda t}}
\exp\!\Bigl(-\frac{x^2}{4\lambda t}\Bigr)
+
\sigma\int_0^t\int_{\R}
\frac{1}{\sqrt{4\pi\lambda(t-\tau)}}
\exp\!\Bigl(-\frac{(x-y)^2}{4\lambda(t-\tau)}\Bigr)\,
\bw(\tau,y)\,dy\,d\tau.
\end{equation}

Since the forcing is additive, $\bz(t,x)$ is Gaussian, and in the
Wiener--Hermite expansion (WHE) \cite{Beltagy1,Beltagy2,Beltagy3}
it is natural to retain only the mean and the first fluctuation term.
Following \cite{Beltagy1,Beltagy2,Beltagy3}, we write
\begin{equation}\label{eq:WHE_alpha1}
\bz(t,x)=\bz_0(t,x)+\int_0^t\int_{-\infty}^{x}
\bz_1(t,x,\tau,\xi)\,H_1(\tau,\xi)\,d\tau\,d\xi,
\end{equation}
where the mean $\bz_0(t,x):=\E[\bz(t,x)]$ and $H_1$ is the first Hermite functional,
which coincides with white noise:
$H_1(t,x)=\bw(t,x)$.
(Under an independence ansatz one may further write $H_1(t,x)=H_1(t)H_1(x)=\bw(t)\bw(x)$,
but this factorization is not needed for the identities below.)
Consequently,
\begin{equation}\label{eq:moments_alpha1}
\E[\bz(t,x)]=\bz_0(t,x),
\qquad
\Var[\bz(t,x)]
=\int_0^t\int_{-\infty}^{x}\bigl|\bz_1(t,x,\tau,\xi)\bigr|^2\,d\tau\,d\xi.
\end{equation}
Substituting \eqref{eq:WHE_alpha1} into \eqref{eq:stoch_heat_alpha1} and using
orthogonality of Hermite functionals yields the fluctuation equation
\begin{equation}\label{eq:fluct_alpha1}
\begin{split}
\partial_t\bz_1(t,x;t_1,x_1)
&=\lambda\,\Delta\bz_1(t,x;t_1,x_1)
+\sigma\,\delta(t-t_1)\delta(x-x_1),\\
\bz_1(0,x;t_1,x_1)&=0,\quad
\lim_{|x|\to\infty}\bz_1(t,x;t_1,x_1)=0,
\end{split}
\end{equation}
cf.\ \cite[Eq.~(4.48)]{Beltagy1}.

The corresponding particular solution is the translated heat kernel
\begin{equation}\label{eq:z1_alpha1}
\bz_1(t,x;t_1,x_1)
=
\frac{\sigma}{\sqrt{4\pi\lambda(t-t_1)}}
\exp\!\Bigl(-\frac{(x-x_1)^2}{4\lambda(t-t_1)}\Bigr),
\qquad t>t_1,
\end{equation}
as in \cite[Eq.~(4.49)]{Beltagy1}.

Therefore,
\begin{equation}\label{eq:var_int_alpha1}
\Var[\bz(t,x)]
=
\sigma^2\int_0^t\int_{\R}
\frac{1}{4\pi\lambda(t-\tau)}
\exp\!\Bigl(-\frac{(x-y)^2}{2\lambda(t-\tau)}\Bigr)\,dy\,d\tau,
\end{equation}
which coincides with \cite[Eq.~(4.50)]{Beltagy1}.
After evaluating the integrals and using the symmetry in $x$, one obtains the closed form
\begin{equation}\label{eq:var_erf_alpha1}
\Var[\bz(t,x)]
=
\frac{\sigma^{2}}{4\lambda\sqrt{\pi t}}\,
\erfc\!\Bigl(\frac{|x|}{4\sqrt{\lambda t}}\Bigr),
\end{equation}
in agreement with \cite[Eq.~(4.51)]{Beltagy1}.

In particular, the variance increases with $t$ and decays as $|x|\to\infty$,
capturing the spatial dissipation of fluctuations.

\begin{figure}[t]
\centering
\includegraphics[width=.55\textwidth]{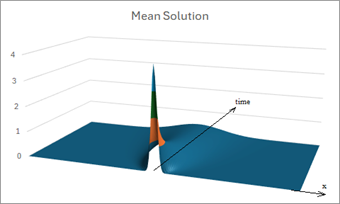}
\caption{Mean solution $\bz_0(t,x)$ for the classical case $\alpha=1$.}
\label{fig:Mean}
\end{figure}

\begin{figure}[t]
\centering
\includegraphics[width=.55\textwidth]{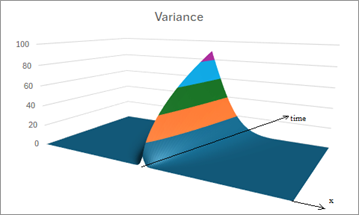}
\caption{Variance $\Var[\bz(t,x)]$ for the classical case $\alpha=1$.}
\label{fig:Var}
\end{figure}


\noindent\underline{\sf Time--fractional stochastic heat equation.}
Let $0<\alpha<1$ and consider
\begin{equation}\label{eq:stoch_heat_frac}
\partial_t^\alpha \bz(t,x)=\lambda\,\Delta\bz(t,x)+\sigma\,\bw(t,x),
\qquad t>0,\; x\in\R,
\end{equation}
with the same initial and boundary conditions as in \eqref{eq:ICBC_classical}.
Using WHE \cite{Beltagy1,Beltagy2,Beltagy3}, the mean field $\bz_0(t,x):=\E[\bz(t,x)]$
solves the deterministic fractional heat equation
\begin{equation}\label{eq:mean_frac_eq}
\partial_t^\alpha \bz_0(t,x)=\lambda\,\Delta\bz_0(t,x),
\qquad
\bz_0(0,x)=\delta_0(x),
\qquad
\lim_{|x|\to\infty}\bz_0(t,x)=0,
\end{equation}
while the first fluctuation kernel satisfies
\begin{equation}\label{eq:fluct_frac_eq}
\begin{split}
\partial_t^\alpha \bz_1(t,x;t_1,x_1)&=\lambda\,\Delta\bz_1(t,x;t_1,x_1)
+\sigma\,\delta(t-t_1)\delta(x-x_1),\\
\bz_1(0,x;t_1,x_1)&=0,\quad
\lim_{|x|\to\infty}\bz_1(t,x;t_1,x_1)=0,
\end{split}
\end{equation}
as stated in \cite[\S4.2]{Beltagy1}.

\medskip

\noindent\underline{\sf Mean solution and the Mainardi function.}
The mean solution admits the Fourier representation
\begin{equation}\label{eq:mean_fourier}
\bz_0(t,x)=\frac1{2\pi}\int_{\R} E_\alpha\!\bigl(-(\lambda k^2)t^\alpha\bigr)\,e^{ikx}\,dk,
\end{equation}
and can be written in closed form via the Mainardi function $M_\alpha$,
\begin{equation}\label{eq:mean_mainardi}
\bz_0(t,x)=\frac{1}{\sqrt{4\pi\lambda\,t^\alpha}}\,
M_\alpha\!\Bigl(\frac{|x|}{\sqrt{\lambda t^\alpha}}\Bigr),
\end{equation}
where
\begin{equation}\label{eq:Mainardi_def}
M_\alpha(u)=\sum_{n=0}^{\infty}\frac{(-1)^n u^{2n}}{(2n)!\,\Gamma(\alpha n-\alpha+1)},
\qquad 0<\alpha<1,
\end{equation}
see \cite[\S4.2]{Beltagy1}. Figure~\ref{fig:Mean-alpha} illustrates the mean solution corresponding to $\alpha=0.6$.
\begin{figure}[H]
    \begin{minipage}{0.4\textwidth}
        \centering
\includegraphics[width=\textwidth]{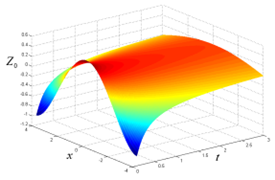}
    \end{minipage}
    \begin{minipage}{0.4\textwidth}
        \centering
\includegraphics[width=\textwidth]{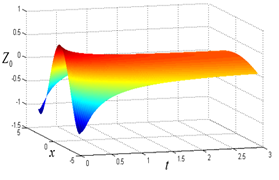}
    \end{minipage}
    \caption{Mean solution $\bz_0(t,x)$, $\alpha=0.6, \lambda=1$.}
        \label{fig:Mean-alpha}
\end{figure}

For the special case $\alpha=\tfrac12$,
\begin{equation}\label{eq:Mainardi_half}
M_{1/2}(u)=\frac{1+u}{\sqrt{\pi}}\,e^{-u^2/4},
\end{equation}
and therefore
\begin{equation}\label{eq:mean_half}
\bz_0(t,x)=
\frac{1}{\sqrt{4\pi\lambda\,t^{1/2}}}
\Bigl(1+\frac{|x|}{\sqrt{4\lambda\,t^{1/2}}}\Bigr)
\exp\!\Bigl(-\frac{x^2}{4\lambda\,t^{1/2}}\Bigr),
\end{equation}
as in \cite[\S4.2]{Beltagy1}. This highlights the subdiffusive scaling $t^{\alpha/2}$ (here $t^{1/4}$) and the slower spreading
compared to the classical case. 

\medskip

\noindent\underline{\sf Fluctuations and variance.}
Using a Green function representation, the first fluctuation kernel is
\begin{equation*}\label{eq:fluct_kernel_frac}
\bz_1(t,x;t_1,x_1)=
\frac{\sigma}{\sqrt{4\pi\lambda\,(t-t_1)^\alpha}}\,
E_{\alpha,\alpha}\!\Bigl(-\frac{(x-x_1)^2}{4\lambda\,(t-t_1)^\alpha}\Bigr),
\qquad 0<t_1<t,
\end{equation*}
see \cite[\S4.2]{Beltagy1}.

By symmetry about $x=0$, for $x>0$ the spatial integration may be restricted to $0<y<x$,
and the variance becomes
\begin{equation*}\label{eq:var_frac_integral}
\Var[\bz(t,x)]
=
\int_0^t\int_{0}^{x}\bigl|\bz_1(t,x;\tau,y)\bigr|^2\,d\tau\,dy
=
\frac{\sigma^2}{4\pi\lambda}\int_0^t\int_0^{x}
(t-\tau)^{-\alpha}
\Biggl[
E_{\alpha,\alpha}\!\Bigl(-\frac{(x-y)^2}{4\lambda\,(t-\tau)^\alpha}\Bigr)
\Biggr]^2\,dy\,d\tau.
\end{equation*}

This representation is well suited for numerical quadrature in $(\tau,y)$.

For $\alpha=\tfrac12$, one may also use the series definition
\begin{equation*}\label{eq:Ehalfhalf_series}
E_{1/2,1/2}(z)=\sum_{k=0}^{\infty}\frac{z^k}{\Gamma\!\left(\frac{k+1}{2}\right)},
\qquad
z=-\frac{(x-y)^2}{4\lambda\,(t-\tau)^{1/2}},
\end{equation*}
leading to an explicit double-series expression for the variance as reported in \cite[\S4.2]{Beltagy1}.

\subsection{Time--fractional heat equation: series representation of the variance}
\label{subsec:series_variance}

For numerical purposes, it is convenient to work with a series form of the variance in the
time--fractional regime. In \cite[\S4.3]{Beltagy1} one arrives at the expansion
\begin{equation*}\label{eq:var_series}
\Var[\bz(t,x)]
=
\frac{\sigma^2}{4\pi\lambda}
\sum_{m=0}^{\infty}
\frac{(-1)^m\,\beta_m\,|x|^{2m+1}}{4^m(2m+1)\,\bigl(1-(m+1)\alpha\bigr)}\,
t^{\,1-(m+1)\alpha},
\end{equation*}
where
\begin{equation*}\label{eq:beta_m}
\beta_m=\sum_{k=0}^{m}\frac{1}{\Gamma(m-k+1)\,\Gamma(\alpha(k+1))}.
\end{equation*}

The series converges for all $\alpha\in(0,1)$ except at the resonance values
\begin{equation*}\label{eq:resonance}
\alpha=\frac{1}{m+1},
\qquad m\in\N,
\end{equation*}
for which the denominator $1-(m+1)\alpha$ vanishes and the variance diverges.

\begin{figure}[H]
\centering
\includegraphics[width=.55\textwidth]{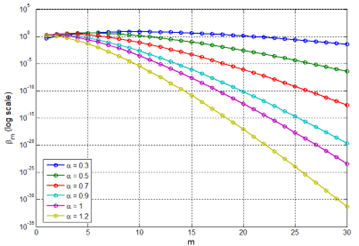}
\caption{Decay of the coefficients $\beta_m$ versus $m$.}
\label{fig:Deacy}
\end{figure}

\begin{figure}[H]
\centering
\includegraphics[width=.55\textwidth]{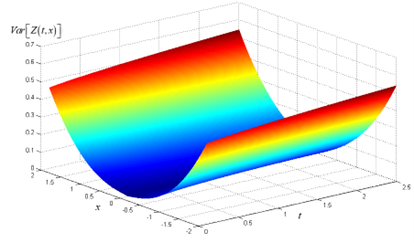}
\caption{Variance $\Var[\bz(t,x)]$ for $\alpha=0.6$ and $\lambda=1$ (e.g.\ $\sigma=1$).}
\label{fig:Var-alpha}
\end{figure}

\subsection{Comments on the figures}
\label{subsec:comments_figures}
Figure~\ref{fig:Mean} depicts the mean solution $\bz_0(t,x)$ for the classical case $\alpha=1$.
As expected, the profile coincides with the Gaussian heat kernel and is symmetric with respect to $x=0$, reflecting the Dirac initial condition and the spatial homogeneity of the Laplacian.
As time increases, the peak at the origin decreases while the solution spreads spatially with the classical diffusive scaling $|x|\sim t^{1/2}$.
This behavior serves as a reference benchmark for comparison with the fractional cases discussed below.

Figure~\ref{fig:Var} shows the variance $\mathrm{Var}[\bz(t,x)]$ in the classical stochastic heat equation.
The variance profile is even in $x$ and attains its maximum at the origin.
As $t$ grows, the variance increases, indicating the accumulation of stochastic fluctuations over time, while spatially it decays as $|x|\to\infty$ due to the smoothing effect of the heat operator.
This illustrates the balance between stochastic forcing and diffusive dissipation in the Markovian case $\alpha=1$.

Figure~\ref{fig:Mean-alpha} presents the mean solution $\bz_0(t,x)$ for a genuinely fractional regime with $\alpha=0.6$.
Compared to Figure~1, the spreading of the profile is significantly slower, which is a hallmark of subdiffusion induced by the Caputo fractional derivative.
The characteristic spatial scale is now of order $t^{\alpha/2}$ rather than $t^{1/2}$, and the profile exhibits heavier tails.
This highlights the strong memory effects encoded in the fractional time derivative and clearly illustrates the departure from classical diffusion.

\paragraph{Figure~\ref{fig:Deacy}: decay of $\beta_m$.}
The coefficients $\beta_m$ in \eqref{eq:var_series} decrease quickly with $m$, becoming very small for
moderately large $m$ (numerically, already around $m\gtrsim 15$), which is favorable for efficient computation
via truncation. This behavior is consistent with absolute convergence of \eqref{eq:var_series} for generic
$\alpha\in(0,1)$, except at the resonant values \eqref{eq:resonance} where the denominators vanish ; see \cite[\S4.4]{Beltagy1}.

Figure~\ref{fig:Var-alpha} exhibits an even profile in $x$, consistent with the symmetric Dirac initial datum
and the spatial homogeneity of the operator. As $t$ increases, the peak at $x=0$ decreases and the mass spreads
in space; for $0<\alpha<1$ this spreading is subdiffusive, with characteristic scale $t^{\alpha/2}$ rather than
$t^{1/2}$. Moreover, the temporal relaxation is slower than in the classical case due to the memory encoded by
the Caputo derivative.

\noindent$\rule[0.05cm]{16.5cm}{0.05cm}$

\noindent{\bf\large Acknowledgements.} {\em The authors express their sincere gratitude to Professor Yuri Luchko for valuable insights into the asymptotic behavior of Mittag–Leffler functions.}

\noindent$\rule[0.05cm]{16.5cm}{0.05cm}$

\noindent{\bf\large Declarations.}
On behalf of all authors, the corresponding author states that there is no conflict of interest. 
No data-sets were generated or analyzed during the current study.

\noindent$\rule[0.05cm]{16.5cm}{0.05cm}$


\end{document}